\documentclass[10pt]{article}
\usepackage{amsmath,amssymb}
\usepackage{enumerate}

\allowdisplaybreaks[1]

\newtheorem{theorem}{Theorem}[section]
\newtheorem{lemma}[theorem]{Lemma}
\newtheorem{proposition}[theorem]{Proposition}

\newtheorem{remark}[theorem]{Remark}
\newtheorem{example}[theorem]{Example}

\newcommand{\Z}{\mathbb{Z}}
\newcommand{\F}{\mathbb{F}}

\renewcommand{\ker}{\operatorname{Ker}}
\newcommand{\id}{\operatorname{id}}

\newcommand{\rad}{\operatorname{rad}}

\newcommand{\Aut}{\operatorname{Aut}}
\newcommand{\Inn}{\operatorname{Inn}}

\newenvironment{proofof}{\par\noindent{\bf Proof}}{$\qed$\par\bigskip}
\newenvironment{proof}{\par\noindent{\bf Proof.}}{$\qed$\par\bigskip}
\newcommand{\qed}{\enspace\vrule  height6pt  width4pt  depth2pt}
\usepackage{color}

\begin{document}

\title{An abundance of simple left braces with abelian multiplicative Sylow subgroups\thanks{The first author was partially supported by the grants
MINECO-FEDER  MTM2017-83487-P and AGAUR 2017SGR1725 (Spain). The
second author is supported in part by Onderzoeksraad of Vrije
Universiteit Brussel and Fonds voor Wetenschappelijk Onderzoek
(Belgium). The third author is supported by the National Science
Centre  grant 2016/23/B/ST1/01045 (Poland).
\newline {\em 2010 MSC:} Primary 16T25, 20D10, 20E22.
\newline {\em Keywords:} Yang-Baxter equation, set-theoretic solution,
brace, simple, $A$-group.}}
\author{F. Ced\'o \and E. Jespers \and J. Okni\'{n}ski}
\date{}

\maketitle

\begin{abstract}
Braces were introduced by Rump to study  involutive
non-de\-gene\-rate set-theoretic solutions of the Yang-Baxter
equation. A constructive method  for producing all such finite
solutions from a description of all finite left braces has been
recently discovered. It is thus a fundamental problem to construct
and classify all simple left braces, as they can be considered as
building blocks for the general theory. This program recently has
been initiated by Bachiller and the authors. In this paper we study
the simple finite left braces such that the Sylow subgroups of their
multiplicative groups are abelian. We provide several new families
of such simple left braces. In particular, they lead to the main,
surprising result, that shows that there is an abundance of such
simple left braces.
\end{abstract}

\section{Introduction}

The quantum Yang-Baxter equation is one of the basic equations of
mathematical physics \cite{Yang} and it has proved to be a
fundamental tool in the theory of quantum groups and related
areas, \cite{K}. In order to describe all involutive
non-degenerate set-theoretic solutions of the Yang-Baxter
equation, a problem posed by Drinfeld \cite{drinfeld}, Rump in
\cite{Rump} introduced the new algebraic object, called a left
brace. A left brace $(B,+,\cdot)$  is a set $B$ with two
operations, $+$ and $\cdot$,
 such that $(B,+)$ is an abelian group, $(B,\cdot)$ is a group and
$a\cdot (b+c)+a=a\cdot b+a\cdot c$ for all $a,b,c\in B$. Braces
have been shown to appear in an intriguing way in several areas of
mathematics, see the surveys \cite{Rump7, C18}. The class of
set-theoretic solutions of the Yang-Baxter equation that are
involutive and non-degenerate has received a lot of attention in
recent years, see for example \cite{CJO2, CJR, ESS, GI, GIC,
GIVdB, Ven1, JObook, Rump1, Rump}.  It has been shown that all
such finite solutions can be obtained, in a constructive way, from
a description of all finite left braces \cite{BCJ}. It thus is a
fundamental problem to classify the building blocks in the class
of all finite left braces, i.e. describe all finite simple left
braces. For finite left braces with nilpotent multiplicative
group, such objects are precisely  the cyclic groups of prime
order, with the operations $\cdot$ and $+$ coinciding (see
\cite[Corollary on page 166]{Rump} and
\cite[Proposition~6.1]{B3}). Attacking the classification problem
from another perspective, Rump in \cite{Rump8} completed the
classification of braces with a cyclic additive group; these
so-called cyclic braces \cite{Rump2} are equivalent to the
T-structures of Etingof et al. \cite{ESS}.

Recall that the multiplicative group $(B,\cdot )$ of any finite left
brace $(B,+,\cdot )$ is solvable  (see \cite[Theorem~2.15]{ESS}
and also \cite[page~167]{Rump}). One of the natural questions that
arose in this context is to determine the orders of finite simple
left braces. In particular in \cite{BCJO} (Problem 5.2) the
following problem is stated: for a given pair of different primes
$p$ and $q$, determine exponents $n$ and $m$ for which there exists
a simple left brace of order $p^nq^m$. A breakthrough came from the
recent result of Bachiller, providing the first nontrivial example
of  a finite simple left  brace \cite{B3}. In \cite{BCJO} and
\cite{BCJO2}, Bachiller and the authors have made progress on
constructing several classes of finite simple  left braces, this via
the matched products and asymmetric products of braces.

The main aim of this paper is to prove the following surprising
theorem that shows that there is an abundance of finite simple left
braces, even when the multiplicative group is metabelian and all its
Sylow subgroups are abelian (the latter are the so called
$A$-groups, \cite{taunt}).

\begin{theorem}\label{MainResult}
 Let $n>1$ be an integer. Let $p_{1}, p_{2}, \dots, p_{n}$ be
distinct primes. There exist positive integers $l_1, l_2, \dots,
l_n$, only depending on $p_{1},p_{2}, \dots , p_n$, such that for
each $n$-tuple   of integers $m_1\geq l_1$, $m_2 \geq l_2, \dots
,m_n\geq l_n$ there exists a simple left brace of order
$p_{1}^{m_1}p_{2}^{m_2}\cdots p_n^{m_{n}}$ that has a metabelian
multiplicative group with abelian Sylow subgroups.
\end{theorem}

The paper is organized as follows. We start in Section 2 with
correcting the proof of a result from \cite{BDG2}, saying that every
finite solvable $A$-group $(B,\cdot)$ admits a structure of a left
brace. We also discuss the structure of finite left braces of this
type, based on the structure of $A$-groups \cite{hall,taunt}. Then a
very general new construction of left braces of the latter type is
presented. These braces $(B,+,\cdot)$ have the additional property
that the multiplicative group $(B, \cdot)$ is metabelian. In Section
3 we construct a new family of finite simple left braces. This
relies on concrete classes of braces presented in Section 2.
Consequently, the main result of the paper,
Theorem~\ref{MainResult}, is proved. Our construction requires
special elements of the Sylow $p$-subgroups of the orthogonal groups
over finite fields of characteristic relatively prime to $p$. The
existence and limitations concerning such elements are discussed in
Section 4.  A problem on prime braces, recently proposed in
\cite{KSV}, Question~4.3, is solved in Section 5. Finally, in
Section 6, we construct another new family of finite simple left
braces using the technique of Section 3.

\section{Finite left braces with abelian multiplicative Sylow subgroups}\label{agroup}

A classical result of P. Hall states that every finite solvable
group $G$ has a Sylow system, that is a family of Sylow subgroups
$\mathcal{S}=\{ S_1,\dots ,S_n\}$ of $G$, one for each prime
dividing the order of $G$, such that $S_iS_j=S_jS_i$ for all $i,j\in
\{1, \dots, n\}$  (see \cite[pages 137--139]{KS}). The system
normalizer of $G$ associated with $\mathcal{S}$ is
$M(G)=\bigcap_{i=1}^n N_G(S_i)$. If the solvable group $G$ is a
normal subgroup of a non necessarily solvable finite group $L$, the
system normalizer of $G$ relative to $L$ associated to $\mathcal{S}$
is $M_L(G)=\bigcap_{i=1}^n N_L(S_i)$. System normalizers  and
relative system normalizers were introduced by P. Hall in
\cite{hall}.

Recall that the multiplicative group $(B,\cdot )$ of any finite left
brace $(B,+,\cdot )$ is solvable. But not all  finite solvable
groups are isomorphic to the multiplicative group of a left brace.
In fact,
 there exist finite $p$-groups that are not
isomorphic to the multiplicative group of any left brace \cite{B2}.
The groups isomorphic to the multiplicative group of a finite left
brace are called involutive Yang-Baxter groups (or simply IYB
groups) in \cite{CJR}.

In \cite[Corollary~4.3]{BDG2} it is stated that any finite solvable
$A$-group is an IYB group.  However, its proof is not correct,
because it is based on the statement that every finite solvable
$A$-group has a normal Sylow subgroup; this is not true, as for
example the group $A_4\times S_3$ is a solvable $A$-group without
Sylow normal subgroups.

First, we give a proof of \cite[Corollary~4.3]{BDG2}.

\begin{theorem} \label{Atype-IYB}
Every finite solvable $A$-group is an IYB group.
\end{theorem}

\begin{proof}
Let $G$ be a finite solvable $A$-group. Let $G=L_0>L_1>\dots
>L_n=1$ be the derived series of $G$. We shall prove the result by
induction on $n$. It is known that every finite abelian group
is IYB  (see for example \cite[Corollary 3.10]{CJR}). Suppose
that $n>1$ and that every finite solvable $A$-group of derived
length $n-1$ is an IYB group. By \cite[(4.5)]{taunt}, $M_G(L_{n-2})$
is a complement in $G$ of $L_{n-1}$, i.e.  $G\cong
L_{n-1}\rtimes M_G(L_{n-2})$. Clearly the derived length of
$M_G(L_{n-2})$ is $n-1$ and thus, by the induction hypothesis, $M_G(L_{n-2})$
is an IYB group. By \cite[Theorem~3.3]{CJR}, $G$ is an IYB group.
\end{proof}

Our aim is to study the structure of finite left braces whose
multiplicative groups are metabelian $A$-groups, with an expectation
that under this natural restriction one can get decisive results.

For the convenience of the reader we recall some basic definitions
and facts that we will use (see \cite{B3,BCJO,CJO2}). Let
$(B,+,\cdot)$ be a left brace. The lambda map of $B$ is the
homomorphism $\lambda\colon (B,\cdot)\longrightarrow \Aut(B,+)$
defined by $\lambda(a)=\lambda_a$ and $\lambda_a(b)=ab-a$ for all
$a,b\in B$. Recall that a left ideal of the left brace $B$ is a
subgroup $L$ of its additive group such that $\lambda_b(a)\in L$
for all $a\in L$ and all $b\in B$. If $L$ is a left ideal of $B$,
then $b^{-1}c=\lambda_{b^{-1}}(c-b)\in L$, for all $b,c\in L$,
thus $L$ also is a subgroup of the multiplicative group of $B$. An
ideal of $B$ is a left ideal $I$ of $B$ such that $I$ is a normal
subgroup of $(B,\cdot)$;  in fact, ideals of left braces are
precisely the kernels of left brace homomorphisms. A nonzero left
brace $B$ is simple if $B$ and $0$ are its only ideals. If $B$ is
finite, then the Sylow subgroups $P_1,\dots, P_n$ of the additive
group of $B$ are left ideals of $B$. Furthermore $\mathcal{S}=\{
P_1,\dots ,P_n\}$ is a Sylow system of the multiplicative group of
$B$. We say that a left brace $B$ is trivial if $a+b=ab$ for all
$a,b\in B$.

Consider a finite left brace $(B,+,\cdot)$ such that $(B,\cdot)$ is
a metabelian $A$-group. Let $P_{1},\ldots, P_{n}$ be the Sylow
subgroups of its additive group. Thus $\mathcal{S}=\{ P_1,\dots
,P_n\}$ is a Sylow system of $(B,\cdot)$. Let $M(B)$ denote its
associated system normalizer, that is $M(B)=
\bigcap_{i=1}^{n}N_{B}(P_{i})$. It is known that the multiplicative
group of $B$ is $B=[B,B]\rtimes M(B)$, a semidirect product and
$M(B)$ is an abelian group, \cite{hall, taunt}.  Then
$[B,B]=T_{1}\times \cdots \times T_{n}$, for subgroups $T_{i}$ of
$(P_{i},\cdot)$, and $M(B)=S_{1}\times \cdots \times S_{n}$ for
subgroups $S_{i}$ of $(P_{i},\cdot)$. Since $(P_i,\cdot)$ is
abelian, $P_{i} = T_{i}\times S_{i}$.

A special case that will be exploited in this paper is when
$(B,+,\cdot)$ is an asymmetric product $B=[B,B]\rtimes_{\circ} M(B)$
of two trivial left braces $[B,B], M(B)$ via a symmetric bilinear
map $b: [B,B]\longrightarrow M(B)$ (on $([B,B],+)$ with values in
$(M(B),+)$) and a homomorphism $$\alpha: (M(B),\cdot)\longrightarrow
\Aut ([B,B],+).$$

Recall that if $T$ and $S$ are two left braces, $b: T\times
T\longrightarrow S$ a  symmetric bilinear map on $(T,+)$ with values
in $(S,+)$, and $\alpha: (S,\cdot)\longrightarrow \Aut(T,+,\cdot)$ a
homomorphism of groups such that
\begin{eqnarray}
&&b(t_2,t_3)=b(\lambda_{t_1}(t_2),\lambda_{t_1}(t_3)),\label{asymmetric1}\\
&&\lambda_s(b(t_2,t_3))=b(\alpha_{s}(t_2),\alpha_{s}(t_3)),\label{asymmetric2}
\end{eqnarray}
for all $t_1,t_2,t_3\in T$ and $s\in S$, then $T\times S$ with the
addition and multiplication given by
$$(t_1,s_1)+(t_2,s_2)=(t_1+t_2,s_1+s_2+b(t_1,t_2)),$$
$$(t_1,s_1)\cdot (t_2,s_2)=(t_1+\alpha_{s_1}(t_2),s_1\cdot s_2)$$
is a left brace called the asymmetric product of $T$ by $S$ (via $b$
and $\alpha$) and denoted by $T\rtimes_{\circ} S$ (see \cite{CCS}).
Note that if $T$ and $S$ are trivial left braces, then condition
(\ref{asymmetric1}) is trivially satisfied and condition
(\ref{asymmetric2}) says that
\begin{equation}\label{asymm3}
\alpha_s\in \mathrm{O}(T,b)
\end{equation}
for all $s\in S$, where by $\mathrm{O}(T,b)$ we denote the
orthogonal group, that is the group consisting of the elements $f\in
\Aut(T,+)$ such that $b(t_1,t_2)=b(f(t_1),f(t_2))$ for all
$t_1,t_2\in T$.

Let $n>1$ be an integer. For each $z\in \Z/(n)$ let $T_{z}$ and
$S_{z}$ be trivial left braces. Let $b_{z}: T_{z}\times T_{z}
\longrightarrow S_{z}$ be a symmetric bilinear map.  For  $z,z'\in
\Z/(n)$ let
 $$f^{(z,z')}: S_{z} \longrightarrow \text{O}(T_{z'},b_{z'})$$
 be a homomorphism of multiplicative groups. The image of $s\in S_z$ we denote by $f^{(z,z')}_{s}$.
 We assume that
\begin{eqnarray}\label{commute}
&&f^{(z_{1},z)}_{s_{1}} f^{(z_2,z)}_{s_{2}} = f^{(z_2,z)}_{s_{2}}
f^{(z_{1},z)}_{s_{1}} \quad\text{and}\quad f^{(z,z)}_{s}=\id_{T_z},
\end{eqnarray}
for all $z_1, z_2,z \in \Z/(n)$, $s_1\in S_{z_1}$, $s_2\in S_{z_2}$
and $s\in S_z$. Let $T=T_1\times \cdots\times T_n$ and $S=S_1\times
\cdots\times S_n$  denote the direct products of the trivial left
braces $T_z$'s and $S_z$'s respectively. Thus $T$ and $S$ are
trivial left braces. Let $b\colon T\times T\longrightarrow S$ be the
symmetric bilinear map defined by
$$b((t_1,\dots ,t_n),(t'_1,\dots ,t'_n))=(b_1(t_1,t'_1),\dots ,b_n(t_n,t'_n)),$$
for all $t_z,t'_z\in T_z$. Let $\alpha\colon
(S,\cdot)\longrightarrow \Aut(T,+)$ be the map defined by
$\alpha(s_1,\dots, s_n)=\alpha_{(s_1,\dots, s_n)}$ and
\begin{eqnarray*}
\alpha_{(s_1,\dots,s_n)}(t_1,\dots,t_n) &=&(f_{s_1}^{(1,1)}\cdots
f_{s_n}^{(n,1)}(t_1),\dots,f_{s_1}^{(1,n)}\cdots
f_{s_n}^{(n,n)}(t_n)),
\end{eqnarray*}
for all $(s_1,\dots,s_n)\in S$ and $(t_1,\dots,t_n)\in T$.

\begin{lemma}   \label{bilinear}
$\alpha$ is a homomorphism of groups and $\alpha_{(s_1,\dots,
s_n)}\in \mathrm{O}(T,b)$ for all $(s_1,\dots ,s_n)\in S$. Therefore
we can construct the asymmetric product $T\rtimes_{\circ} S$ of $T$
by $S$ via $b$ and $\alpha$.
\end{lemma}

\begin{proof}
By (\ref{commute}), it is clear that $\alpha$ is a
homomorphism of groups. Let $(s_1,\dots,s_n)\in S$ and $(t_1,\dots
,t_n),(t'_1,\dots ,t'_n)\in T$. We have
\begin{eqnarray*}
\lefteqn{b(\alpha_{(s_1,\dots,s_n)}(t_1,\dots,t_n),\alpha_{(s_1,\dots,s_n)}(t'_1,\dots,t'_n))}\\
 &=&b((f_{s_1}^{(1,1)}\cdots
f_{s_n}^{(n,1)}(t_1),\dots,f_{s_1}^{(1,n)}\cdots
f_{s_n}^{(n,n)}(t_n)),\\
&&\quad (f_{s_1}^{(1,1)}\cdots
f_{s_n}^{(n,1)}(t'_1),\dots,f_{s_1}^{(1,n)}\cdots
f_{s_n}^{(n,n)}(t'_n)))\\
 &=&(b_1(f_{s_1}^{(1,1)}\cdots
f_{s_n}^{(n,1)}(t_1),f_{s_1}^{(1,1)}\cdots
f_{s_n}^{(n,1)}(t'_1)),\\
&&\quad \dots,b_n(f_{s_1}^{(1,n)}\cdots
f_{s_n}^{(n,n)}(t_n),f_{s_1}^{(1,n)}\cdots
f_{s_n}^{(n,n)}(t'_n)))\\
 &=&(b_1(t_1,t'_1),\dots,b_n(t_n,t'_n)).
\end{eqnarray*}
Hence $\alpha_{(s_1,\dots,s_n)}\in \mathrm{O}(T,b)$ for all
$(s_1,\dots ,s_n)\in S$. Thus the result follows.
\end{proof}
For $z\in \Z/(n)$, let $A_z=\{ ((t_1,\dots ,t_n),(s_1,\dots
,s_n))\in T\rtimes_{\circ}S\mid t_{z'}=0 \text{ and }s_{z'}=0 \text{
for all }z'\neq z \}$.

\begin{proposition}    \label{construction1}
Each $A_z$ is a left ideal of $T\rtimes_{\circ}S$ with abelian
multiplicative group. The additive group of $T\rtimes_{\circ}S$ is
the direct sum of the additive groups of the left ideals $A_z$. If
$\alpha$ is not trivial, then the multiplicative group of
$T\rtimes_{\circ}S$ is  metabelian but not abelian. Furthermore, if
the left ideals $A_z$ are finite  and  of relatively prime orders,
then the multiplicative group of $T\rtimes_{\circ}S$ also is an
$A$-group.
\end{proposition}

\begin{proof}
Note that, for all $u,v\in T$ and $s,s'\in S$,
$(u+v,s+s')-(u,s)=(v,s'-b(u,v))$ by the definition of the sum in
the asymmetric product. Thus we have
\begin{eqnarray*}
\lefteqn{\lambda_{((t_1,\dots ,t_n),(s_1,\dots, s_n))}((t'_1,\dots
,t'_n),(s'_1,\dots, s'_n))}\\
&=&((t_1,\dots ,t_n),(s_1,\dots, s_n))((t'_1,\dots
,t'_n),(s'_1,\dots, s'_n))\\
&&-((t_1,\dots ,t_n),(s_1,\dots, s_n))\\
&=&((t_1,\dots ,t_n)\cdot\alpha_{(s_1,\dots,s_n)}(t'_1,\dots,
t'_n),\\
&&\; (s_1,\dots
,s_n)\cdot(s'_1,\dots, s'_n))-((t_1,\dots ,t_n),(s_1,\dots, s_n))\\
&=&((t_1+f_{s_1}^{(1,1)}\cdots f_{s_n}^{(n,1)}(t'_1),\dots
,t_n+f_{s_1}^{(1,n)}\cdots f_{s_n}^{(n,n)}(t'_n)),\\
&&\; (s_1+s'_1,\dots
,s_n+s'_n))-((t_1,\dots ,t_n),(s_1,\dots, s_n))\\
&=&((f_{s_1}^{(1,1)}\cdots f_{s_n}^{(n,1)}(t'_1),\dots
,f_{s_1}^{(1,n)}\cdots f_{s_n}^{(n,n)}(t'_n)),(s'_1,\dots
,s'_n)\\
&&\quad -b((t_1,\dots ,t_n),(f_{s_1}^{(1,1)}\cdots
f_{s_n}^{(n,1)}(t'_1),\dots
,f_{s_1}^{(1,n)}\cdots f_{s_n}^{(n,n)}(t'_n))))\\
&=&((f_{s_1}^{(1,1)}\cdots f_{s_n}^{(n,1)}(t'_1),\dots
,f_{s_1}^{(1,n)}\cdots
f_{s_n}^{(n,n)}(t'_n)),\\
&&\; (s'_1-b_1(t_1,f_{s_1}^{(1,1)}\cdots
f_{s_n}^{(n,1)}(t'_1)),\dots ,s'_n-b_n(t_n,f_{s_1}^{(1,n)}\cdots
f_{s_n}^{(n,n)}(t'_n)))),
\end{eqnarray*}
for all $(t_1,\dots ,t_n),(t'_1,\dots ,t'_n)\in T$ and $(s_1,\dots,
s_n),(s'_1,\dots ,s'_n)\in S$.  Hence it is clear that
$$\lambda_{((t_1,\dots ,t_n),(s_1,\dots, s_n))}((t'_1,\dots
,t'_n),(s'_1,\dots, s'_n))\in A_z,$$ for all $((t_1,\dots
,t_n),(s_1,\dots, s_n))\in T\rtimes_{\circ}S$ and $((t'_1,\dots
,t'_n),(s'_1,\dots, s'_n))\in A_z$. Let $((t_1,\dots,
,t_n),(s_1,\dots,s_n)),((t'_1,\dots,t'_n),(s'_1,\dots,s'_n))\in
A_z$. By (\ref{commute}), we have that
\begin{eqnarray*}\lefteqn{((t_1,\dots,t_n),(s_1,\dots,s_n))\cdot
((t'_1,\dots,t'_n),(s'_1,\dots,s'_n))}\\
&=&((0,\dots, 0,t_z+t'_z,0,\dots,0),(0,\dots,
0,s_z+s'_z,0,\dots,0))\in A_z.
\end{eqnarray*}
Hence $A_z$ is a left ideal of $T\rtimes_{\circ}S$ with abelian
multiplicative group.

Note that the multiplicative group of $T\rtimes_{\circ}S$ is the
semidirect product of two abelian groups via the action $\alpha$.
Hence, if $\alpha$ is not trivial, this group is metabelian  and not
abelian.

Since the multiplicative group of each $A_z$ is abelian, the last
assertion also follows.
\end{proof}

\begin{remark}{\rm
Assume that $(B,+,\cdot)$ is a finite simple left brace and
$(B,\cdot)$ is a metabelian $A$-group. If one of the Sylow subgroups
of $(B,\cdot)$ is cyclic, then $B$ is a cyclic (and trivial) left
brace. Indeed, let $B_{1},\ldots , B_{n}$ be the Sylow subgroups of
the additive group of $B$. Then $\mathcal{S}=\{ B_{1},\ldots ,
B_{n}\}$ is a Sylow system of $(B,\cdot)$. Say, $B_1$ is cyclic
$p_1$-group, for some prime $p_1$, and suppose that $n>1$.  Write
the multiplicative group of $B$ as $B= [B,B] M(B)=[B,B]\rtimes
M(B)$. Let $S_1$ be the Sylow $p_1$-subgroup of $[B,B]$ and let
$S_2$ be the Sylow $p_1$-subgroup of $M(B)$. Then $S_1S_2$ is an
inner direct product because $(B,\cdot)$ is an $A$-group and it is a
Sylow $p_1$-subgroup of $B$. Hence, $B_1$ being cyclic implies that
$B_1\subseteq [B,B]$ or $B_1\subseteq M(B)$.  If $B_{1}\subseteq
[B,B]$, then $B_{1}$ is a normal subgroup of $(B,\cdot)$ because
$[B,B]$ is abelian and thus has a unique Sylow subgroup of order
$|B_1|$. Hence, $B_{1}$ is an ideal of $B$, a contradiction. On the
other hand, if $B_{1}\subseteq M(B )$ then $B_{2}\cdots B_{n}$ is a
subgroup normalized by $B_{1}$. Hence $B_{2}\cdots B_{n}$ is a
normal subgroup of $B$, hence an ideal of $B$, a contradiction
again.}
\end{remark}

\section{The main new construction} \label{construction}
In this section, using the construction of left braces introduced
in Section~\ref{agroup}, we shall construct a new family of simple
left braces, based on the existence of special elements of prime
order $p$ in the orthogonal group $\mathrm{O}(V,b)$ of a
non-singular symmetric bilinear form $b$ over a finite dimensional
$\Z/(q)$-vector space $V$, for a prime $q\neq p$. Using this
construction of simple left braces, we shall prove the main result
of this paper, Theorem~\ref{MainResult}.

To check the simplicity of a left brace the following lemma is
useful.

\begin{lemma}\label{ideal}
\cite[Lemma~2.5]{BCJO} If $I$ is an ideal of a left brace $B$, then
$(\lambda_b-\id)(a)\in I$, for all $a\in B$ and $b\in I$.
\end{lemma}

First, we introduce and fix some notation. Let $n>1$ be an integer.
For $z\in \Z/(n)$, let $p_z$ be a prime number and let $r_{z}$  be a
positive integer. We assume that $p_{z}\neq p_{z'}$ for $z\neq z'$.

Let $m_z$ be positive integers such that $m_z \geq\max\{
r_{z},r_{z-1}\}$ for all $z\in \Z/(n)$. Let $V_z$ be a
$\Z/(p_z)$-vector space and let $b_z$  be a non-singular symmetric
bilinear form over $V_z$. Assume also that there exist elements $f_z
\in O(V_z, b_z)$ (the associated orthogonal group) of order
$p_{z-1}$.

For each $z\in \Z/(n)$, consider the trivial left braces
$T_z=V_z^{m_z}$ and $S_z=(\Z/(p_z))^{r_z}$. Note that $S_z$ also is
a $\Z/(p_z)$-vector space. The standard basis elements of
$(\Z/(p_z))^{r_z}$ we denote by $e_{1}^{(z)}, \dots ,
e_{r_{z}}^{(z)}$. Let $b'_z:T_z\times T_z\longrightarrow S_z$ be the
symmetric bilinear map defined by
$$b'_z((u_1,\dots,u_{m_z}),(v_1,\dots,v_{m_z}))=\sum_{i=1}^{r_z-1}b_z(u_i,v_i)e_{i}^{(z)}+\sum_{j=1}^{m_z}b_z(u_j,v_j)e_{r_z}^{(z)},$$
for all $(u_1,\dots,u_{m_z}),(v_1,\dots,v_{m_z})\in T_z$ (we agree
that the first sum is zero if $r_z=1$). Let
$f^{(z-1,z)}:S_{z-1}\longrightarrow \Aut(T_z,+)$ be the map
defined by $f^{(z-1,z)}(s)=f_s^{(z-1,z)}$ and
\begin{eqnarray*}\lefteqn{f_s^{(z-1,z)}(u_1,\dots,u_{m_z})}\\
&=&(f_{z}^{\mu_1+\mu_{r_{z-1}}}(u_1),\dots,
f_{z}^{\mu_{r_{z-1}-1}+\mu_{r_{z-1}}}(u_{r_{z-1}-1}),\\
&&\quad f_{z}^{\mu_{r_{z-1}}}(u_{r_{z-1}}),\dots,
f_{z}^{\mu_{r_{z-1}}}(u_{m_{z}})),
\end{eqnarray*}
for all $s=\sum_{i=1}^{r_{z-1}}\mu_ie_{i}^{(z-1)}\in S_{z-1}$ and
$(u_1,\dots,u_{m_z})\in T_z$  (we agree that if $r_{z-1}=1$, then
$f_s^{(z-1,z)}(u_1,\dots,u_{m_z})=(f_{z}^{\mu_{1}}(u_{1}),\dots,
f_{z}^{\mu_{1}}(u_{m_{z}}))$).

\begin{lemma}\label{orthof}
For each $z\in \Z/(n)$, the map $f^{(z-1,z)}:S_{z-1}\longrightarrow
\Aut(T_z,+)$ is a homomorphism of multiplicative groups and
$f_s^{(z-1,z)}\in \mathrm{O}(T_{z},b'_z)$, for all $s\in S_{z-1}$.
\end{lemma}

\begin{proof} Recall that $S_{z-1}$ is a trivial left brace. Thus
the addition and the multiplication of $S_{z-1}$ coincide. Hence it
is clear that $f^{(z-1,z)}:S_{z-1}\longrightarrow \Aut(T_z,+)$ is a
homomorphism of multiplicative groups. Let
$s=\sum_{i=1}^{r_{z-1}}\mu_ie_{i}^{(z-1)}\in S_{z-1}$. By
assumption, $f_z \in O(V_z, b_z)$. Hence we have
\begin{eqnarray*}
\lefteqn{b'_z(f_s^{(z-1,z)}(u_1,\dots,u_{m_z}),f_s^{(z-1,z)}(v_1,\dots,v_{m_z}))}\\
&=&b'_z ((f_{z}^{\mu_1+\mu_{r_{z-1}}}(u_1),\dots,
f_{z}^{\mu_{r_{z-1}-1}+\mu_{r_{z-1}}}(u_{r_{z-1}-1}),\\
&&\quad f_{z}^{\mu_{r_{z-1}}}(u_{r_{z-1}}),\dots,
f_{z}^{\mu_{r_{z-1}}}(u_{m_{z}})),\\
&&\qquad
(f_{z}^{\mu_1+\mu_{r_{z-1}}}(v_1),\dots,
f_{z}^{\mu_{r_{z-1}-1}+\mu_{r_{z-1}}}(v_{r_{z-1}-1}),\\
&&\quad f_{z}^{\mu_{r_{z-1}}}(v_{r_{z-1}}),\dots,
f_{z}^{\mu_{r_{z-1}}}(v_{m_{z}})))\\
&=&\sum_{i=1}^{r_z-1}b_z(u_i,v_i)e_i^{(z)}+\sum_{j=1}^{m_z}b_z(u_j,v_j)e_{r_z}^{(z)}\\
&=&b'_z((u_1,\dots, u_{m_z}),(v_1,\dots, v_{m_z})),
\end{eqnarray*}
where the second equality holds because $f_z\in
\mathrm{O}(V_z,b_z)$. Thus the result follows.
\end{proof}

Let $T=T_1\times \cdots\times T_n$ and $S=S_1\times\cdots\times S_n$
be the direct  products of the left braces $T_z$'s and $S_z$'s
respectively. Let $b\colon T\times T\longrightarrow S$ be the
symmetric bilinear map defined by
$$b((t_1,\dots ,t_n),(t'_1,\dots ,t'_n))=(b'_1(t_1,t'_1),\dots ,b'_n(t_n,t'_n)),$$
for all $(t_1,\dots ,t_n),(t'_1,\dots ,t'_n)\in T$. Let
$f^{(z,z')}\colon S_z\longrightarrow \Aut(T_z,+)$ be the trivial
homomorphism for all $z'\neq z+1$.  Note that the homomorphisms
$f^{(z,z')}$ satisfy (\ref{commute}). Let $\alpha\colon
(S,\cdot)\longrightarrow \Aut(T,+)$ be the map defined by
$\alpha(s_1,\dots,s_n)=\alpha_{(s_1,\dots,s_n)}$ and
\begin{eqnarray*}
\alpha_{(s_1,\dots,s_n)}(t_1,\dots ,t_n)&=&(f^{(1,1)}_{s_1}\cdots
f^{(n,1)}_{s_n}(t_1),\dots ,f^{(1,n)}_{s_1}\cdots
f^{(n,n)}_{s_n}(t_n))\\
&=&(f^{(n,1)}_{s_n}(t_1),f^{(1,2)}_{s_1}(t_2),\dots
,f^{(n-1,n)}_{s_{n-1}}(t_n)),
\end{eqnarray*}
for all $(s_1,\dots,s_n)\in S$ and $(t_1, \dots, t_n)\in T$. By
Lemma~\ref{orthof}, and Lemma~\ref{bilinear}, we have that
$\alpha_{(s_1,\dots,s_n)}\in \mathrm{O}(T,b)$, for all
$(s_1,\dots,s_n)\in S$ and we can construct the asymmetric product
$T\rtimes_{\circ} S$ of $T$ by $S$ via $b$ and $\alpha$.

For $z\in \Z/(n)$, let $A_z=\{ ((t_1,\dots ,t_n),(s_1,\dots
,s_n))\in T\rtimes_{\circ}S\mid t_{z'}=0 \text{ and }s_{z'}=0 \text{
for all }z'\neq z \}$. By Proposition~\ref{construction1},  $A_z$ is
a left ideal of $T\rtimes_{\circ}S$, in fact it is the Sylow
$p_z$-subgroup of the additive group of $T\rtimes_{\circ}S$.

\begin{theorem} \label{simple0}
The asymmetric product $T\rtimes_{\circ} S$ is a simple left brace
if and only  if $f_{z}-\id$ is bijective for all $z\in \Z/(n)$.
\end{theorem}

\begin{proof}
To prove the necessity of the stated condition it is sufficient
(and easy) to check that $J=\{ ((t_1,\dots ,t_n),(s_1,\dots
,s_n))\in T\rtimes_{\circ} S \mid t_z =(u_{1}, \dots , u_{m_z})$
such that $u_{i} \in \text{Im}(f_{z}-\id)  \}$ is an ideal of
$T\rtimes_{\circ} S$.

For the converse, suppose that  $f_{z}-\id$ is bijective for all
$z\in \Z/(n)$. Let $I$ be a nonzero ideal of $T\rtimes_{\circ}S$ and
let $0\neq ((t_1,\dots ,t_n),(s_1,\dots ,s_n))\in I$.  Since the
$p_{z}$'s are distinct prime numbers, we may assume, without loss of
generality, that for some $k\in \Z/(n)$ we have $(t_k,s_k)\neq 0$
and $(t_z,s_z)=0$ for all $z\neq k$. Write $t_{k}=( u_{1}, \dots,
u_{m_{k}})$ and $s_k= \sum_{i=1}^{r_{k}} \mu_{i} e_{i}^{(k)}$. Thus
$((t_1,\dots ,t_n),(s_1,\dots ,s_n))\in I\cap A_k$.

For $z\in \Z/(n)$, let $a_{z,i}$ be the element in $A_{z}$
defined by
 $$a_{z,i} = ((0,\dots,0),(s'_1,\dots ,s'_n)),$$
with $s'_z=e_{i}^{(z)}$.

Note that, for all $t,t'\in T$ and all $s,s'\in S$,
$$(t,
s)-(t', s')=(t-t', s-s'-b(t-t',t'))$$ by the definition of the sum
in the asymmetric product $T\rtimes_{\circ} S$.  Throughout the
proof we will use this formula several times.

 {\it Claim 1}: If there exists a positive integer $j$ such
that $r_{k} \leq j \leq m_{k}$ and $u_{j}\neq 0$, then
  $a_{k,r_{k}}\in I$.\\
Indeed, since, by assumption,  $b_{k}$ is non-singular, there
exists $u\in V_{k}$ such that $b_{k}(u,u_j)\neq 0$. Let $c_{k}\in
A_{k}$ be the element
 $$c_{k} = ((t'_1,\dots,t'_n),(0,\dots ,0)),$$
with $t'_k=(u'_{1}, \dots , u'_{m_{k}})$  such that $u'_{j}=u$ and
$u'_{l} =0$ for all $l\neq j$. Then,
\begin{eqnarray*}
\lefteqn{\lambda_{c_k}((t_1,\dots ,t_n),(s_1,\dots ,s_n))}\\
&=&((t'_1,\dots,t'_n),(0,\dots,0))((t_1,\dots,t_n),(s_1,\dots,s_n))\\
&&-((t'_1,\dots,t'_n),(0,\dots,0))\\
&=&((t'_1+t_1,\dots,t'_n+t_n),(s_1,\dots,s_n))-((t'_1,\dots,t'_n),(0,\dots,0))\\
&=&((t_1,\dots,t_n),(s_1,\dots,s_n)-b((t'_1,\dots,t'_n),(t_1,\dots,t_n)))\\
&=&((t_1,\dots,t_n),(s_1-b'_1(t'_1,t_1),\dots,s_n-b'_n(t'_n,t_n)))\in
I\cap A_k,
 \end{eqnarray*}
with
$$s_k-b'_k(t'_k,t_k)
=s_k -b_{k}(u,u_{j})
  e_{r_{k}}^{(k)},
$$
(recall that $(t_z,s_z-b'_z(t'_z,t_z)) =0$ for all $z\neq k$).
Now,
\begin{eqnarray*}
\lefteqn{((t_1,\dots ,t_n),(s_1,\dots ,s_n))}\\
&&-((t_1,\dots
,t_n),(s_1-b'_1(t'_1,t_1),\dots,s_n-b'_n(t'_n,t_n)))\\
&=&((0,\dots ,0),(b'_1(t'_1,t_1),\dots ,b'_n(t'_n,t_n)))\in I\cap
A_k,
\end{eqnarray*}
with $b'_1(t'_k,t_k)= b_{k}(u,u_{j}) e_{r_{k}}^{(k)}$. Thus
$a_{k,r_k}\in I$ and the claim follows.

{\it Claim 2}: If $u_j\neq 0$ for some $j$ with $1\leq j < r_{k}$,
then  $a_{k,j} +a_{k,r_{k}}\in I$.\\
As in the proof of Claim 1, we see that
\begin{eqnarray*}
\lefteqn{((t_1,\dots ,t_n),(s_1,\dots
,s_n))-\lambda_{c_k}((t_1,\dots
,t_n),(s_1,\dots ,s_n))}\\
&=&((0,\dots ,0),(b'_1(t'_1,t_1),\dots ,b'_n(t'_n,t_n)))\in I\cap
A_k,
\end{eqnarray*}
with $b'_k(t'_k,t_k)= b_{k}(u,u_{j})(e_{j}^{(k)}
+e_{r_{k}}^{(k)})$ and $b_{k}(u,u_{j})\neq 0$. Thus
$a_{k,j}+a_{k,r_k}\in I$ and the claim follows.

{\it Claim 3}: If $\mu_{r_k}=0$ and $\mu_{i}\neq 0$ for some $i$
with $1\leq i <r_{k}$ then $c_{k+1}\in I$ for some  $c_{k+1}\in
A_{k+1}$ such that $c_{k+1} =((t'_1,\dots,
t'_n),(s'_1,\dots,s'_n))$, $t'_{k+1}=(v_{1}, \dots , v_{m_{k}+1})$,
 and some $v_{l}\neq 0$.\\
Indeed,  since $f_{k+1}$ has order $p_k$ and $\mu_i\neq 0$, there
exists $v\in V_{k+1}$ such that $f_{k+1}^{\mu_{i}}(v)\neq
 v$. Let
  $$c'_{k+1} =((t''_1,\dots ,t''_n),(0,\dots,0))\in A_{k+1}$$
 be the element with $t''_{k+1}=(v'_{1}, \dots , v'_{m_{k}+1})$
 such that $v'_{i} =v$ and $v'_{l} =0$ for all $l\neq i$.
 By Lemma~\ref{ideal}, we have that
 \begin{eqnarray*}
 \lefteqn{(\lambda_{((t_1,\dots,t_n),(s_1, \dots , s_n))}-\id)(c'_{k+1})}\\
 &=&((t_1,\dots,t_n),(s_1, \dots , s_n))((t''_1,\dots ,t''_n),(0,\dots,0))\\
 &&-((t_1,\dots,t_n),(s_1, \dots , s_n))-((t''_1,\dots
 ,t''_n),(0,\dots,0))\\
 &=&((t_1+f_{s_n}^{(n,1)}(t''_1),\dots,t_n+f_{s_{n-1}}^{(n-1,n)}(t''_n)),(s_1, \dots , s_n))\\
 &&-((t_1+t''_1,\dots,t_n+t''_n),(s_1+b'_1(t_1,t''_1), \dots , s_n+b'_n(t_n,t''_n)))\\
 &=&((f_{s_n}^{(n,1)}(t''_1)-t''_1,\dots,f_{s_{n-1}}^{(n-1,n)}(t''_n)-t''_n),\\
 &&\quad (-b'_1(t_1,t''_1)-b'_1(f_{s_n}^{(n,1)}(t''_1)-t''_1,t_1+t''_1),\\
&&\quad  \dots ,
-b'_n(t_n,t''_n)-b'_n(f_{s_{n-1}}^{(n-1,n)}(t''_n)-t''_n,t_n+t''_n)))\in
 I\cap A_{k+1}
 \end{eqnarray*}
 and
 $$f_{s_{k}}^{(k,k+1)} (t''_{k+1}) -t''_{k+1} =
  (0,\dots , 0, f_{k+1}^{\mu_{i}} (v) -v, 0, \dots , 0).$$
Hence Claim 3 follows.

{\it Claim 4}: If $\mu_{r_{k}}\neq 0$ then $c_{k+1}\in I$ for some
$c_{k+1}\in A_{k+1}$ such that
$c_{k+1}=((t'_1,\dots, t'_n),(s'_1,\dots,s'_n))$, $t'_{k+1}=(v_{1}, \dots , v_{m_{k+1}})$,
 and some $v_{l}\neq 0$.\\
As in the proof of Claim 3, there exists $v\in V_{k+1}$ such that
$f_{k+1}^{\mu_{r_k}}(v)\neq v$. Let
  $$c'_{k+1} =((t''_1,\dots ,t''_n),(0,\dots, 0))\in A_{k+1}$$
 be the element with $t''_{k+1}=(v'_{1}, \dots , v'_{m_{k}+1})$
 such that $v'_{r_k} =v$ and $v'_{l} =0$ for all $l\neq r_k$.
 As in the proof of Claim 3, we have that
 \begin{eqnarray*}
 \lefteqn{(\lambda_{((t_1,\dots,t_n),(s_1, \dots , s_n))}-\id)(c'_{k+1})}\\
 &=&((f_{s_n}^{(n,1)}(t''_1)-t''_1,\dots,f_{s_{n-1}}^{(n-1,n)}(t''_n)-t''_n),\\
 &&\quad (-b'_1(t_1,t''_1)-b'_1(f_{s_n}^{(n,1)}(t''_1)-t''_1,t_1+t''_1),\\
&&\quad  \dots ,
-b'_n(t_n,t''_n)-b'_n(f_{s_{n-1}}^{(n-1,n)}(t''_n)-t''_n,t_n+t''_n)))\in
 I\cap A_{k+1}
 \end{eqnarray*}
 and
 $$f_{s_{k}}^{(k,k+1)} (t''_{k+1}) -t''_{k+1} =
  (0,\dots , 0, f_{k+1}^{\mu_{r_k}} (v) -v, 0, \dots , 0).$$
Hence Claim 4 follows.

We are now in a position to prove the sufficiency of the
conditions. By Claims 1, 2, 3 and 4,  without loss of generality,
we may assume that either
 $a_{k,r_{k}}\in I$ or $a_{k,j}+a_{k,r_{k}}\in I$
for some $1\leq j < r_{k}$.

Suppose first that $a_{k,r_{k}} \in I$. Then by Lemma~\ref{ideal},
for all $t'_{k+1} =(v_1, \dots , v_{m_{k+1}})\in T_{k+1}$, we get
that
\begin{eqnarray*}
\lefteqn{(\lambda_{a_{k,r_{k}}}-\id) ((0,\dots , 0, t'_{k+1}, 0,\dots , 0),(0,\dots , 0))}\\
 &=&
 ((0,\dots, 0, f_{e_{r_{k}}^{(k)}}^{(k,k+1)}(t'_{k+1})-t'_{k+1}, 0, \dots , 0),(0,\dots,s'_{k+1},0,\dots, 0))\in
 I\cap A_{k+1}
\end{eqnarray*}
for some $s'_{k+1}\in S_{k+1}$ and
$$
f_{e_{r_{k}}^{(k)}}^{(k,k+1)}(t'_{k+1}) -t'_{k+1} =
((f_{k+1}-\id)(v_{1}), \dots , (f_{k+1}-\id)(v_{m_{k+1}})).
$$
 Since, by assumption, $f_{k+1} -\id$ is bijective, by Claim
1 and Claim 2, we get that
 $$
 a_{k+1,r_{k+1}},\;  a_{k+1,j}+ a_{k+1,r_{k+1}} \in I
 $$
for all $1\leq j < r_{k+1}$.  Thus, all $a_{k+1,j}\in I\cap
A_{k+1}$ (for $1\leq j \leq r_{k+1}$). Consequently,
$$((0,\dots, 0),(0,\dots,0,s'_{k+1},0,\dots ,0))\in I\cap
A_{k+1}.$$ Therefore
\begin{eqnarray*} \lefteqn{((0,\dots, 0,
f_{e_{r_{k}}^{(k)}}^{(k,k+1)}(t'_{k+1})-t'_{k+1}, 0, \dots ,
0),(0,\dots, s'_{k+1},0,\dots, 0))}\\
&&-((0,\dots, 0),(0,\dots,0,s'_{k+1},0,\dots ,0))\\
&=& ((0,\dots, 0, f_{e_{r_{k}}^{(k)}}^{(k,k+1)}(t'_{k+1})-t'_{k+1},
0, \dots , 0),(0,\dots,0))\in I\cap A_{k+1}.
\end{eqnarray*}
Since $$ f_{e_{r_{k}}^{(k)}}^{(k,k+1)}(t'_{k+1}) -t'_{k+1} =
((f_{k+1}-\id)(v_{1}), \dots , (f_{k+1}-\id)(v_{m_{k+1}}))
$$
and  $f_{k+1} -\id$ is bijective,  in this case,
 $$ A_{k+1}\subseteq I.$$

 Now, consider the second case, namely suppose that $a_{k,j}+a_{k,r_{k}}\in I$ for some $1\leq j < r_{k}$.
 Then by Lemma~\ref{ideal},  for all $t'_{k+1} =(v_1, \dots , v_{m_{k+1}})\in T_{k+1}$,
 \begin{eqnarray*}
 \lefteqn{
 (\lambda_{a_{k,j}+a_{k,r_{k}}}-\id) ((0,\dots , 0,t'_{k+1}, 0,\dots , 0),(0,\dots, 0))}\\
 &=&
 ((0,\dots , 0, f_{{e_j}^{(k)}+e_{r_{k}}^{(k)}}^{(k,k+1)} (t'_{k+1})-t'_{k+1}, 0,\dots, 0),(0,\dots ,0,s'_{k+1},0,\dots, 0))\in
 I\cap A_{k+1}
 \end{eqnarray*}
for some $s'_{k+1}\in S_{k+1}$ and
\begin{eqnarray*}
\lefteqn{f_{{e_j}^{(k)}+e_{r_{k}}^{(k)}}^{(k,k+1)} (t'_{k+1})-t'_{k+1} }\\
&=&
( (f_{k+1} -\id)(v_{1}), \dots ,  (f_{k+1} -\id)(v_{j-1}),  (f^{2}_{k+1} -\id)(v_{j}),  (f_{k+1} -\id)(v_{j+1}),\\
&&
 \dots ,  (f_{k+1} -\id)(v_{m_{k+1}})).
\end{eqnarray*} Since $m_{k+1}\geq r_k>j$ and $f_{k+1}\neq\id$, there exists
$v_{m_{k+1}}\in V_{k+1}$ such that $(f_{k+1}
-\id)(v_{m_{k+1}})\neq 0$, thus Claim 1 yields that
 $$
 a_{k+1,r_{k+1}}\in I.
 $$
 Therefore,  applying the previous case, we get that $A_{k+2}\subseteq I$.

Now,  combining the above two cases we clearly get that
$A_{z}\subseteq I$, for all $z\in \Z/(n)$. So,
$I=T\rtimes_{\circ}S$, as desired.
\end{proof}

We now prove the main result of the paper, Theorem~\ref{MainResult}.
\bigskip

\begin{proofof} {\bf (of Theorem~\ref{MainResult}).}
For each $z\in \Z/(n)$ let $V_z=(\Z/(p_z))^{2(p_{z-1}-1)}$ and let
$b_z\colon V_z\times V_z\longrightarrow \Z/(p_z)$ be the symmetric
bilinear form with associated matrix
$$\left(
\begin{array}{c|c}
0&I_{p_{z-1}-1}\\
\hline
I_{p_{z-1}-1}&0
\end{array}\right)$$
with respect the standard basis of $V_z$, where $I_m$ denotes the
$m\times m$ identity matrix. Let $C_z$ be the companion matrix of
the polynomial $x^{p_{z-1}-1}+x^{p_{z-1}-2}+\cdots +x+1\in
\Z/(p_z)[x]$, that is
$$C_z=\left(
\begin{array}{ccccc}
0&0&\ldots&0&-1\\
1&0&\ldots&0&-1\\
0&\ddots&\ddots&\vdots&\vdots\\
\vdots&\ddots&\ddots&0&-1\\
0&\ldots&0&1&-1
\end{array}\right)\in GL_{p_{z-1}-1}(\Z/(p_z)).$$
Let $f_z$ be the automorphism of $V_z$ with associated matrix
$$\left(
\begin{array}{c|c}
C_z&0\\
\hline 0&(C_z^{-1})^t
\end{array}\right).$$
 (For a matrix $A$, we denote by $A^t$ the transpose of $A$.)
Note that $f_z$ has order $p_{z-1}$ and
\begin{eqnarray*}\lefteqn{\left(
\begin{array}{c|c}
0&I_{p_{z-1}-1}\\
\hline I_{p_{z-1}-1}&0
\end{array}\right)}\\
&=& \left(
\begin{array}{c|c}
C_z&0\\
\hline 0&(C_z^{-1})^t
\end{array}\right)^t\left(
\begin{array}{c|c}
0&I_{p_{z-1}-1}\\
\hline I_{p_{z-1}-1}&0
\end{array}\right)\left(
\begin{array}{c|c}
C_z&0\\
\hline 0&(C_z^{-1})^t
\end{array}\right)
.
\end{eqnarray*}
Hence $f_z\in \mathrm{O}(V_z,b_z)$ and $f_z-\id$ is bijective.

 Let $n'_z =\max\{\dim(V_z),\dim(V_{z-1})\}$ and let $l_{z}
=n'_z (\dim(V_z)+1)$. Let $m_z\geq l_z$. Write $m_z
=m'_{z}\dim(V_z)+ r_z$ with $0< r_z \leq \dim(V_z)$ and a positive
integer $m'_{z}$. Note that $m_z\geq l_z=n'_z(\dim(V_z)+1)=
n'_z\dim(V_z)+n'_z\geq n'_z\dim(V_z)+r_z$. Hence $m'_z\geq n'_z\geq
\max\{ r_z,r_{z-1}\}$. Let $T_z =V_{z}^{m'_{z}}$ and
$S_z=(\Z/(p_{z}))^{r_{z}}$. Let $T=T_1\times \cdots\times T_n$ and
$S=S_1\times \cdots\times S_n$. Then, using Theorem~\ref{simple0}
and Proposition~\ref{construction1}, one can construct a simple left
brace $T\rtimes_{\circ}S$ that has a metabelian multiplicative group
with abelian Sylow subgroups, and
$$|T\rtimes_{\circ}S|=p_1^{m'_1\dim(V_1)+r_1}p_2^{m'_2\dim(V_2)+r_2}\cdots p_n^{m'_n\dim(V_n)+r_n}=p_1^{m_1}p_2^{m_2}\cdots p_n^{m_n}.$$
Therefore the result follows.
\end{proofof}

\begin{remark}\label{dimension}{\rm
In the proof of the main result (Theorem~\ref{MainResult}), we have
seen that the positive integers $l_1,\dots, l_n$ of the statement
can be defined as
$$l_z=\max\{\dim(V_z),\dim(V_{z-1})\}(\dim(V_z)+1),$$
for all $z\in\Z/(n)$, where $V_z$ is a finite dimensional
$\Z/(p_z)$-vector space with a non-singular symmetric bilinear form
$b_z$ and $f_z\in \mathrm{O}(V_z,b_z)$ of order $p_{z-1}$ such that
$f_z-\id$ is bijective.}
\end{remark}

\section{Search for the best bounds}\label{bounds}

In this section we study the  minimal possible orders of finite
simple left braces with multiplicative $A$-group that  can arise
from the construction in Section 3.

In the construction of simple left braces  in  Section 3 we use, for
each pair of distinct primes $p$ and $q$, a $\Z/(p)$-vector space
$V$ with a non-singular symmetric bilinear form $b$ and an element
$f\in \mathrm{O}(V,b)$ of order $q$ such that $f-\id$ is bijective.
By Remark~\ref{dimension}, in order to discover sharp lower bounds
for the numbers $l_z$ used in the main result
Theorem~\ref{MainResult},  a natural question is to determine the
minimal dimension of such a vector space $V$.

Note that for $q=2$, we can take $V=\Z/(p)$ and the symmetric
bilinear form $b\colon V\times V\longrightarrow \Z/(p)$ defined by
$b(a,b)=ab$, for all $a,b\in\Z/(p)$. In this case, the element $f\in
\mathrm{O}(V,b)$ defined by $f(a)=-a$, for all $a\in \Z/(p)$, has
order $2$ and $f-\id$ is bijective. Thus the minimal dimension, in
this case, is $1$.

 Another motivation for this study is the following problem
stated in \cite[Problem 5.2]{BCJO}: for which prime numbers $p,q$
and positive integers $n,m$ there exists a simple left brace of
order $p^nq^m$?

In the next remark we explain that there are some obvious
restrictions on $n$ and $m$.

\begin{remark}
{\rm  If  $B$ is a simple left brace of order $p_1^{m_1}p_2^{m_2}$,
where $p_1,p_2$ are distinct prime numbers and $m_1,m_2$ are
positive integers, then   the Sylow subgroups of the multiplicative
group of $B$  are not normal in $(B,\cdot)$ (see
\cite[Proposition~6.1]{B3}). Hence, by the Sylow Theorems, $p_1$
should be a divisor of $p_2^t-1$,  for some $1\leq t\leq m_2$, and
$p_2$ should be a divisor of $p_1^s-1$, for some $1\leq s\leq m_1$.
Hence if $k_1$ is the order of $p_1$ in $(\Z/(p_2))^*$ and $k_2$ is
the order of $p_2$ in $(\Z/(p_1))^*$, then $k_i\leq m_i$, for
$i=1,2$.

 By \cite[Remark~5.3]{BCJO}, if $p$ is an odd prime and $k$ is
the order of $2$ in $(\Z/(p))^*$, then any left brace of order
$2^kp$ is not simple.}
\end{remark}

For the construction of simple left braces as explained in Section 3
we use elements of prime order $p$ in the orthogonal group over
$\Z/(q)$, for a prime $q\neq p$. We now first recall what is needed
about the orders of these orthogonal groups.

Let $p$ be an odd prime number. Following \cite{Wi}, we denote
by $GO_{2m+1}(p)$ the orthogonal group (unique up to isomorphism) of
a $(2m+1)$-dimensional vector space over the field of $p$ elements
with respect to a non-singular symmetric bilinear form (equivalently
a non-degenerate quadratic form), for any positive integer $m$.
There are two non-isomorphic orthogonal groups of a $2m$-dimensional
vector space $V$ over the field of $p$ elements with respect to a
non-singular symmetric bilinear form $b$ denoted by $GO^+_{2m}(p)$
(if $V$ has a totally isotropic subspace of dimension $m$) and
$GO^-_{2m}(p)$ (if $V$ has no totally isotropic subspace of
dimension $m$). The orders of these groups are the following:
$$|GO_{2m+1}(p)|=2p^{m^2}(p^2-1)(p^4-1)\cdots (p^{2m}-1),$$
$$|GO^+_{2m}(p)|=2p^{m(m-1)}(p^2-1)(p^4-1)\cdots (p^{2m-2}-1)(p^m-1)$$
and
$$|GO^-_{2m}(p)|=2p^{m(m-1)}(p^2-1)(p^4-1)\cdots
(p^{2m-2}-1)(p^m+1).$$

For the prime $2$ there are two classes of non-singular symmetric
bilinear forms: alternating and non-alternating. Alternating
non-singular symmetric bilinear forms are over $2m$-dimensional
$\Z/(2)$-vector spaces   and  their orthogonal groups coincide with
the symplectic group $\mathrm{Sp}_{2m}(2)$ that has order
$$|\mathrm{Sp}_{2m}(2)|=2^{m^2}(2^2-1)(2^4-1)\cdots (2^{2m}-1),$$
(see \cite{Wi}). The orthogonal group $\mathrm{O}(m,2)$ of a
non-alternating non-singular symmetric bilinear form $b$ over an
$m$-dimensional $\Z/(2)$-vector space $V$ has order
$$|\mathrm{O}(2t+1,2)|=2^{t^2}(2^2-1)(2^4-1)\cdots (2^{2t}-1),$$
if $m=2t+1$, and
$$|\mathrm{O}(2t,2)|=2^{t^2}(2^2-1)(2^4-1)\cdots (2^{2(t-1)}-1),$$
if $m=2t$ (see \cite{MW}).

\begin{remark}\label{ortho}{\rm
 Let $p_1$ be an odd prime number. Let $p\neq p_1$ be a prime
number. Let $k$ be the order of $p$ in $(\Z/(p_1))^*$.  If $p$ and
$k$ are odd, then
\begin{itemize}
\item[(i)] $p_1$ is a divisor of $|GO_{2m+1}(p)|$ if and only if
$k\leq m$,
\item[(ii)] $p_1$ is a divisor of $|GO^+_{2m}(p)|$ if and only if
$k\leq m$,
\item[(iii)] $p_1$ is a divisor of $|GO^-_{2m}(p)|$ if and only if
$k\leq m-1$.
\end{itemize}
If $p$ is odd and $k$ is even, then
\begin{itemize}
\item[(iv)] $p_1$ is a divisor of $|GO_{2m+1}(p)|$ if and only if
$k\leq 2m$,
\item[(v)] $p_1$ is a divisor of $|GO^+_{2m}(p)|$ if and only if
$k\leq 2m-2$,
\item[(vi)] $p_1$ is a divisor of $|GO^-_{2m}(p)|$ if and only if
$k\leq 2m$.
\end{itemize}
If $p=2$ and $k$ is odd, then
\begin{itemize}
\item[(vii)] $p_1$ is a divisor of $|\mathrm{Sp}_{2m}(2)|$ if and only if
$k\leq m$,
\item[(viii)] $p_1$ is a divisor of $|\mathrm{O}(2t+1,2)|$ if and only if
$k\leq t$,
\item[(ix)] $p_1$ is a divisor of $|\mathrm{O}(2t,2)|$ if and only if
$k\leq t-1$.
\end{itemize}
 If $p=2$ and  $k$ is even, then
\begin{itemize}
\item[(x)] $p_1$ is a divisor of $|\mathrm{Sp}_{2m}(2)|$ if and only if
$k\leq 2m$,
\item[(xi)] $p_1$ is a divisor of $|\mathrm{O}(2t+1,2)|$ if and only if
$k\leq 2t$,
\item[(xii)] $p_1$ is a divisor of $|\mathrm{O}(2t,2)|$ if and only if
$k\leq 2(t-1)$.
\end{itemize}}
\end{remark}

\begin{remark}\label{decomposition}{\rm
Let $p_1,p_2$ be two distinct prime numbers. Let $k_1$ be the
order of $p_1$ in $(\Z/(p_2))^*$.  Note that the field
$\F_{p_1^{k_1}}$ of $p_1^{k_1}$ elements is the smallest field of
characteristic $p_1$ that contains an element of multiplicative
order $p_2$. Consider $x^{p_2}-1\in \Z/(p_1)[x]$ and its
decomposition as a product of monic irreducible polynomials
$$x^{p_2}-1=(x-1)q_1(x)\cdots q_s(x).$$
Note that $\F_{p_1^{k_1}}$ contains all roots of $q_j(x)$ for all
$j$. Let $\alpha\in \F_{p_1^{k_1}}$ be a root of $q_j(x)$. Then the
subfield $\left( \Z/(p_1)\right) (\alpha)$ of $\F_{p_1^{k_1}}$ is $\F_{p_1^{k_1}}$
by the minimality of $\F_{p_1^{k_1}}$. Furthermore,
$$\F_{p_1^{k_1}}=\left( \Z/(p_1)\right) (\alpha)\cong(\Z/(p_1)[x])/(q_j(x)).$$
Therefore the degree of $q_j(x)$ is $k_1$, for all $j=1,\dots ,s$.
Note that $q_1(x),\dots ,q_s(x)$ are $s$ distinct monic
polynomials.}
\end{remark}

For a positive integer $k$ we define
\begin{eqnarray} \label{nu}
\nu(k)&=& \left\{\begin{array}{ll}
1&\quad\mbox{if $k$ is even,}\\
2&\quad\mbox{if $k$ is odd.}
\end{array}\right.
\end{eqnarray}

\begin{theorem}\label{minimal}
Let $p_1$ be an odd prime number. Let $p\neq p_1$ be a prime number.
Let $k$ be the order of $p$ in $(\Z/(p_1))^*$.  Then the minimal
dimension of a $\Z/(p)$-vector space $V$ with a non-singular
bilinear form $b$ and $f\in \mathrm{O}(V,b)$ of order $p_1$ such
that $f-\id$ is bijective is $\nu(k)k$. Furthermore,
\begin{itemize}
\item[(i)] if $p$ and $k$ are odd, then $\mathrm{O(V,b)}\cong
GO^+_{2k}(p)$,
\item[(ii)] if $p$ is odd and $k$ is even, then $\mathrm{O}(V,b)\cong
GO^-_{k}(p)$,
\item[(iii)] if $p=2$ and $k$ is odd, then
$\mathrm{O}(V,b)\cong \mathrm{Sp}_{2k}(2)$, and
\item[(iv)] if $p=2$ and $k$ is even, then $\mathrm{O}(V,b)\cong \mathrm{Sp}_{k}(2)$.
\end{itemize}
\end{theorem}

\begin{proof}
First suppose that $k$ is odd. By Remark~\ref{ortho}, if $p$ is odd,
then
\begin{itemize}
\item[(i)] $p_1$ is a divisor of $|GO_{2m+1}(p)|$ if and only if
$k\leq m$,
\item[(ii)] $p_1$ is a divisor of $|GO^+_{2m}(p)|$ if and only if
$k\leq m$,
\item[(iii)] $p_1$ is a divisor of $|GO^-_{2m}(p)|$ if and only if
$k\leq m-1$.
\end{itemize}
By Remark~\ref{ortho}, if $p=2$, then
\begin{itemize}
\item[(vii)] $p_1$ is a divisor of $|\mathrm{Sp}_{2m}(2)|$ if and only if
$k\leq m$,
\item[(viii)] $p_1$ is a divisor of $|\mathrm{O}(2t+1,2)|$ if and only if
$k\leq t$,
\item[(ix)] $p_1$ is a divisor of $|\mathrm{O}(2t,2)|$ if and only if
$k\leq t-1$.
\end{itemize}

Let $V$ be a $\Z/(p)$-vector space of dimension $2k$ and let $b$ be
a non-singular symmetric bilinear form such that either
$\mathrm{O(V,b)}\cong GO^+_{2k}(p)$, if $p$ is odd, or
$\mathrm{O}(V,b)\cong \mathrm{Sp}_{2k}(2)$, if $p=2$. Let $f\in
\mathrm{O}(V,b)$ be an element of order $p_1$. We shall prove that
$f-\id$ is bijective.

Suppose that $\ker(f-\id)\neq 0$. Since $\dim(V)=2k$ and $f\neq
\id$, by Remark~\ref{decomposition}, the minimal polynomial of $f$
is $m_{f}(x) =(x-1)q(x)$, for some monic irreducible polynomial
$q(x)\in \Z/(p)[x]$ of degree $k$, which is a divisor of
$x^{p_{1}}-1$. Hence $V=\ker(f-\id)\oplus \ker(q(f))$. Let
$\vec{u}\in \ker(q(f))$ be a nonzero element. Consider the linear
span $W$ of $\{ \vec{u}, f(\vec{u}),\dots ,f^{k-1}(\vec{u})\}$. By
Remark~\ref{ortho}, $W$ is  singular. Therefore, the radical $I$ of
the restriction of $b$ to $W$ is a nonzero subspace of $W$ invariant
by $f$, and the restriction of $f$ to $I$ has minimal polynomial
$q(x)$. Hence $I$ has dimension $k$, and then $I=W$ is a totally
isotropic subspace invariant by $f$, of dimension $k$. Note that
$\dim(\ker(q(f)))<2k$. Hence $W=\ker(q(f))$. Let $\vec{v}\in
\ker(f-\id)$. Then for every $\vec{w}\in W$, we have
$$b(\vec{w},\vec{v})=b(f(\vec{w}),f(\vec{v}))=b(f(\vec{w}),\vec{v}).$$
Therefore $b((f-\id)(\vec{w}),\vec{v})=0$. Since the restriction of
$f-\id$ to $W$ is bijective, we have that $b(\vec{w},\vec{v})=0$,
for all $\vec{w}\in W$ and all $v\in \ker(f-\id)$. But then
$W\subseteq\rad(b)$, a contradiction because $b$ is non-singular.
Therefore $\ker(f-\id)=0$ and the result follows in this case.

Suppose now that $k$ is even. By Remark~\ref{ortho}, if $p$ is odd,
then
\begin{itemize}
\item[(iv)] $p_1$ is a divisor of $|GO_{2m+1}(p)|$ if and only if
$k\leq 2m$,
\item[(v)] $p_1$ is a divisor of $|GO^+_{2m}(p)|$ if and only if
$k\leq 2m-2$,
\item[(vi)] $p_1$ is a divisor of $|GO^-_{2m}(p)|$ if and only if
$k\leq 2m$.
\end{itemize}
By Remark~\ref{ortho}, if $p=2$, then
\begin{itemize}
\item[(x)] $p_1$ is a divisor of $|\mathrm{Sp}_{2m}(2)|$ if and only if
$k\leq 2m$,
\item[(xi)] $p_1$ is a divisor of $|\mathrm{O}(2t+1,2)|$ if and only if
$k\leq 2t$,
\item[(xii)] $p_1$ is a divisor of $|\mathrm{O}(2t,2)|$ if and only if
$k\leq 2(t-1)$.
\end{itemize}

Let $V$ be a $\Z/(p)$-vector space of dimension $k$ and let $b$ be a
non-singular symmetric bilinear form such that either
$\mathrm{O(V,b)}\cong GO^-_{k}(p)$, if $p$ is odd, or
$\mathrm{O}(V,b)\cong \mathrm{Sp}_{k}(2)$, if $p=2$. Let $f\in
\mathrm{O}(V,b)$ be an element of order $p_1$. We shall prove that
$f-\id$ is bijective.

Since $\dim(V)=k$ and $f\neq \id$, again by
Remark~\ref{decomposition}, the minimal polynomial of $f$ is a monic
irreducible polynomial of degree $k$, which is a divisor of
$x^{p_{1}}-1$ in $\Z/(p)[x]$. Hence $f-\id$ is bijective in this
case.

Therefore the result follows.
\end{proof}

\begin{remark}\label{min} {\rm
 Let $n>1$ be an integer. For $z\in \Z/(n)$, let $p_z$ be a prime
number. We assume that $p_{z}\neq p_{z'}$ for $z\neq z'$. Let
$k_z$ be the order of $p_{z}$ in $(\Z/(p_{z-1}))^*$. Let $m_z$ be
positive integers. Let $V_z$ be a  $\Z/(p_z)$-vector space of
dimension either $\dim(V_z)=\nu(k_z)k_z$, if $p_{z-1}\neq 2$, or
$\dim(V_z)=1$, if $p_{z-1}=2$. By Theorem~\ref{minimal}, there
exist a non-singular symmetric bilinear form $b_z$ over $V_z$ and
$f_z\in \mathrm{O}(V_z,b_z)$ such that $f_z$ has order $p_{z-1}$
and $f_z-\id$ is bijective. Let $r_z=1$ for all $z\in\Z/(n)$.
Then, as in Section 3 and by Theorem~\ref{simple0}, we can
construct a simple left brace $T\rtimes_{\circ} S$ of order
$p_1^{m_1\dim(V_1)+1}\cdots p_n^{m_n\dim(V_n)+1}$ such that its
multiplicative group is a metabelian $A$-group. In particular, for
$n=2$, $p_1$ and $p_2$ odd prime numbers and $m_1=m_2=1$, we
obtain simple left braces of order
$p_1^{\nu(k_1)k_1+1}p_2^{\nu(k_2)k_2+1}$ and also of order
$p_1^22^{2k+1}$, where $k$ is the order of $2$ in $(\Z/(p_1))^*$.

Furthermore, in our main result, Theorem~\ref{MainResult}, by
Remark~\ref{dimension}, we can take $l_z=\max\{\nu(k_z)k_z,
\nu(k_{z-1})k_{z-1}\}(\nu(k_z)k_z+1)$, if $p_{z-1}\neq 2$, and
$l_z=2\nu(k_{z-1})k_{z-1}$, if $p_{z-1}= 2$, where $k_z$ is the
order of $p_{z}$ in $(\Z/(p_{z-1}))^*$, for all $z\in \Z/(n)$.}
\end{remark}

In view of the above, the following result provides the best
bounds on the exponents showing up in Theorem~\ref{MainResult}
that can be obtained using the approach of this paper.

\begin{theorem}\label{main2}
Let $n>1$ be an integer. For $z\in \Z/(n)$, let $p_z$ be a prime
number. We assume that $p_{z}\neq p_{z'}$ for $z\neq z'$. Let $k_z$
be the order of $p_{z}$ in $(\Z/(p_{z-1}))^*$. Let
$$l_z=\left\{
\begin{array}{ll}
\max\{\nu(k_z)k_z, \nu(k_{z-1})k_{z-1}\}(\nu(k_z)k_z+1)&\text{ if }
p_{z-1}\neq 2\\
2\nu(k_{z-1})k_{z-1}&\text{ if }
p_{z-1}= 2\\
\end{array}\right.$$
Then for each $n$-tuple of integers $m_1\geq l_1$, $m_2 \geq l_2,
\dots ,m_n\geq l_n$ there exists a simple left brace of order
$p_{1}^{m_1}p_{2}^{m_2}\cdots p_n^{m_{n}}$ that has a metabelian
multiplicative group with abelian Sylow subgroups.
\end{theorem}

\begin{remark}{\rm
 In Remark~\ref{min} we have seen the minimal orders of the form
$p^{n}q^{m}$, for distinct prime numbers $p,q$, of the left simple
braces constructed as in Section 3. We do not know whether these are
the minimal possible orders of simple left braces such that their
multiplicative groups are metabelian $A$-groups. But without this
restriction on the multiplicative group we can obtain simple left
braces of smaller order.

For example, consider the primes $3$ and $7$. The order of $7$ in
$(\Z/(3))^*$ is $1$ and the order of $3$ in $(\Z/(7))^*$ is $6$.
Using the construction of Section 3, we can construct simple left
braces of orders $3^{6m_1+r_1}7^{2m_2+r_2}$ for all positive
integers $m_1,m_2,r_1,r_2$ such that $m_i \geq r_j$, for all
$i,j$. The smallest simple left brace of this form has order
$3^{7}7^{3}$.

Since $3|(7-1)$ we can also construct simple left braces of order
$3^{n}7^{m}$ using \cite[Subsection~5.3]{BCJO2}.  Indeed, with this
method one can construct simple left braces of orders
$3^{6mn+1}7^n$, for all positive integers $m,n$, with multiplicative
group isomorphic to
$$((\Z/(3))^{6mn}\rtimes (\Z/(7))^n)\rtimes \Z/(3)$$
and of derived length $3$. The Sylow $7$-subgroups of such groups
are abelian, but the Sylow $3$-subgroups of such groups are not
abelian (they are metabelian). Note that we obtain new possible
orders with this method: $3^{6m+1}\cdot 7$ and $3^{12m+1}\cdot 7^2$,
for every positive integer $m$. In particular, there is a
simple left brace of order $3^7\cdot 7$.}
\end{remark}

\section{Finite prime left braces with multiplicative $A$-group}

Konovalov, Smoktunowicz and Vendramin introduced prime left braces
(see \cite{KSV}).  Recall that in a left brace $B$, the operation
$*$ is defined by $a*b=ab-a-b$, and for ideals $I$ and $J$, $I*J$
denotes the additive subgroup of $B$ generated by $\{a*b\mid a\in
I,\; b\in J\}$. By an analogy with ring theory, a left brace $B$ is
prime if $I*J\neq 0$ for all non-zero ideals $I$ and $J$ of $B$. It
is clear that every non-trivial simple left brace is prime. In
\cite[Question~4.4]{KSV}  the following question is posed: do there
exist finite prime non-simple left braces?  Using the construction
of left braces as introduced in Section~3 we  will show that the
answer to this question is affirmative. Therefore this also gives an
affirmative answer to \cite[Question~4.3]{KSV}.

Recall that if $A,B$ are left braces and $\alpha\colon
(B,\cdot)\longrightarrow \Aut(A,+,\cdot)$ is a homomorphism of
groups then the semidirect product of left braces $A\rtimes B$ via
$\alpha$ is the asymmetric product (as defined in
Section~\ref{agroup}) with zero bilinear map $b$. So, this is a left
brace with multiplicative group the semidirect product of the
multiplicative groups of $A$ and $B$, and with addition defined
componentwise (see \cite{Rump5} or \cite[p. 113 ]{CJO2}).

We denote the subgroup of the inner automorphisms of a group $G$ by
$\Inn(G)$.

\begin{proposition}\label{prime}
Let $A$ be any non-trivial finite simple left brace. Suppose that
there exists an element
$\alpha\in\Aut(A,+,\cdot)\setminus\Inn(A,\cdot)$ of prime order
$p$. Then the semidirect product $B=A\rtimes \Z/(p)$, via the
action $z\mapsto \alpha^{z}$, for $z\in\Z/(p)$, is a prime
non-simple left brace.
\end{proposition}

\begin{proof}
First, we claim that the only ideals of $B$ are $\{ (0,0)\}$,
$A\times\{ 0\}$ and $B$. Let $I$ be a nonzero ideal of $B$. If
$I\cap (A\times\{ 0\})$ is nonzero, then, since $A$ is a simple left
brace, we have  that $A\times \{ 0\}\subseteq I$, and therefore
either $I=A\times\{ 0\}$ or $I=B$. Suppose that $I\cap(A\times\{
0\})=\{(0,0)\}$. Let $(a,b)\in I$ be a nonzero element.  By
Lemma~\ref{ideal} we have that
$$(a,b)(c,d)-(a,b)-(c,d)=(a\alpha^b(c)-a-c,0)=(\lambda_{(a,b)}-\id)(c,d)\in I,$$
and
$$(c,d)(a,b)-(c,d)=(c\alpha^d(a)-c,b)=\lambda_{(c,d)}(a,b)\in I,$$
for all $(c,d)\in B$. Since $I\cap (A\times\{ 0\})=\{(0,0)\}$ and
$(a,b)\neq (0,0)$, it follows that $b\neq 0$, $a\alpha^b(c)=a+c$
and $c\alpha^d(a)=a+c$, for all $(c,d)\in B$. Hence
$a\alpha^b(c)=c\alpha^{d}(a)$ for every $(c,d)\in B$, so that
$a\alpha^b(c)=ca$, for all $c\in A$. Since $b$ is a nonzero
element in $\Z/(p)$ and $\alpha$ has order $p$, we have that
$\alpha\in\Inn(A,\cdot)$, in contradiction with the choice of
$\alpha$. Thus the only ideals of $B$ are indeed: $\{ (0,0)\}$,
$A\times\{ 0\}$ and $B$. Since $A$ is a non-trivial simple left
brace, $A*A=A$. Since $(A\times\{ 0\})*(A\times\{
0\})=(A*A)\times\{ 0\}=A\times\{ 0\}$, $B$ is prime and the result
follows.
\end{proof}

\begin{remark}{\rm  Let $A$ be a left brace. Note that
$\Aut(A,+,\cdot)$ is clearly a subgroup of $\Aut(A,\cdot)$ and also
is a subgroup of $\Aut(A,+)$. It is easy to check that
$\Inn(A,\cdot)$ is contained in $\Aut(A,+,\cdot)$ exactly when $A$
is a two-sided brace.  Indeed, let $B=\{ a\in A\mid
a^{-1}(b+c)a=a^{-1}ba+a^{-1}ca, \text{ for all } b,c\in A\}$. Note
that
\begin{eqnarray*}B&=&\{ a\in A\mid
a^{-1}(b+c)a=a^{-1}ba+a^{-1}ca, \text{ for all } b,c\in
A\}\\
&=&\{ a\in A\mid (b+c)a=a^{-1}(ab+ac-a)a=ba+ca-a, \text{ for all }
b,c\in A\}.
\end{eqnarray*}
It is clear that $B$ is a subgroup of the multiplicative group of
$A$. Let $a,b\in B$ and $c,d\in A$. We have
\begin{eqnarray*}(c+d)(a-b)&=&(c+d)a-(c+d)b+c+d\\
&=&ca+da-a-cb-db+b+c+d\\
&=&c(a-b)+d(a-b)-(a-b).\end{eqnarray*} Hence, $a-b\in B$ and thus
$B$ is a subbrace of $A$. Note that $B$ is a two-sided brace.

Note also that if $A$ is finite and $\Inn(A,\cdot)$ is
contained in $\Aut(A,+,\cdot)$, then every Sylow subgroup of $(A,+)$
is a normal Sylow subgroup of $(A,\cdot)$ and therefore $(A,\cdot)$
is nilpotent.}
\end{remark}

 We conclude with a concrete example of a finite prime non-simple
left brace,  based on  Proposition~\ref{prime}.

\begin{example}
There is a finite prime non-simple left brace $B$ with
multiplicative $A$-group.
\end{example}

\begin{proof}
First we construct a finite simple left brace using the main
construction of Section 3 and Theorem~\ref{simple0}.

Let $V_1=(\Z/(2))^2$ and $V_2=\Z/(3)$ be vector spaces over $\Z/(2)$
and $\Z/(3)$ respectively. Let $b_1\colon V_1\times
V_1\longrightarrow \Z/(2)$ be the symmetric bilinear form defined by
$$b_1((x_1,x_2),(y_1,y_2))=x_1y_2+x_2y_1,$$
for all $(x_1,x_2),(y_1,y_2)\in V_1$. Let $b_2\colon V_2\times
V_2\longrightarrow \Z/(3)$ be the symmetric bilinear  form defined
by
$$b_2(z,z')=zz',$$
for all $z,z'\in V_2$. Let $f_1$ be the automorphism of $V_1$
defined by
$$f_1(x_1,x_2)=(x_2,x_1+x_2),$$
for all $(x_1,x_2)\in V_1$. It is easy to check that $f_1\in
\mathrm{O}(V_1,b_1)$, it has order $3$ and $f_1-\id$ is bijective.
Let $f_2$ be the automorphism of $V_2$ defined by
$$f_2(z)=-z,$$
for all $z\in V_2$. It is easy to check that $f_2\in
\mathrm{O}(V_2,b_2)$, it has order $2$ and $f_2-\id$ is bijective.
Now, using the notation of Section 3, we take $n=2$, $m_1=5$,
$m_2=r_1=r_2=1$. Consider the trivial left braces $T_1=V_1^5$,
$T_2=V_2$, $S_1=\Z/(2)$ and $S_2=\Z/(3)$. Let $T=T_1\times T_2$
and $S=S_1\times S_2$ be the direct products of these left braces.
As in Section 3, we define $b'_1\colon T_1\times
T_1\longrightarrow S_1$ by
$$b'_1((u_1,u_2,u_3,u_4,u_5),(u'_1,u'_2,u'_3,u'_4,u'_5))=\sum_{j=1}^5b_1(u_j,u'_j),$$
for all $(u_1,u_2,u_3,u_4,u_5),(u'_1,u'_2,u'_3,u'_4,u'_5)\in T_1$,
and $b'_2\colon T_2\times T_2\longrightarrow S_2$ by
$b'_2(v,v')=b_2(v,v')$, for all $v,v'\in T_2$.  Let $f^{(1,2)}\colon
S_1\longrightarrow \Aut(T_2,+)$ be the map defined by
$f^{(1,2)}(s_1)=f^{(1,2)}_{s_1}$ and
$$f^{(1,2)}_{s_1}(v)=f_2^{s_1}(v),$$
for all $s_1\in S_1$ and $v\in T_2$. Let $f^{(2,1)}\colon
S_2\longrightarrow \Aut(T_1,+)$ be the map defined by
$f^{(2,1)}(s_2)=f^{(2,1)}_{s_2}$ and
$$f^{(2,1)}_{s_2}(u_1,u_2,u_3,u_4,u_5)=(f_1^{s_2}(u_1),f_1^{s_2}(u_2),f_1^{s_2}(u_3),f_1^{s_2}(u_4),f_1^{s_2}(u_5)),$$
for all $s_2\in S_2$ and $(u_1,u_2,u_3,u_4,u_5)\in T_1$. By
Lemma~\ref{orthof}, $f^{1,2}_{s_1}\in \mathrm{O}(T_2,b'_2)$ and
$f^{(2,1)}_{s_2}\in \mathrm{O}(T_1,b'_1)$. Let $b\colon T\times
T\longrightarrow S$ be the symmetric bilinear map defined by
$$b((t_1,t_2),(t'_1,t'_2))=(b'_1(t_1,t'_1),b'_2(t_2,t'_2)),$$
for all $(t_1,t_2),(t'_1,t'_2)\in T$. Let $\alpha\colon
S\longrightarrow \Aut(T,+)$ be the map defined by
$\alpha(s_1,s_2)=\alpha_{(s_1,s_2)}$ and
$$\alpha_{(s_1,s_2)}(t_1,t_2)=(f^{(2,1)}_{s_2}(t_1),f^{(1,2)}_{s_1}(t_2)),$$
for all $(s_1,s_2)\in S$ and $(t_1,t_2)\in T$. By
Theorem~\ref{simple0}, the asymmetric product $T\rtimes_{\circ}S$
via $b$ and $\alpha$ is a simple left brace.

Consider the trivial brace $\Z/(5)$ and $f\in\Aut(T_1,\cdot)$
defined by
$$f(u_1,u_2,u_3,u_4,u_5)=(u_2,u_3,u_4,u_5,u_1),$$
for all $(u_1,u_2,u_3,u_4,u_5)\in T_1$.  Note that $f\in
\mathrm{O}(T_1,b'_1)$ and $f^{(2,1)}_{s_2}f=ff^{(2,1)}_{s_2}$, for
all $s_2\in S_2$. Let $\beta\colon \Z/(5)\longrightarrow
\Aut(T\rtimes_{\circ}S,+,\cdot )$ be the map defined by
$\beta(a)=\beta_a$ and
$$\beta_a((t_1,t_2),(s_1,s_2))=((f^{a}(t_1),t_2),(s_1,s_2)),$$
for all $a\in \Z/(5)$ and $((t_1,t_2),(s_1,s_2))\in
T\rtimes_{\circ}S$. Note that
\begin{eqnarray*}
\lefteqn{\beta_a(((t_1,t_2),(s_1,s_2))+((t'_1,t'_2),(s'_1,s'_2)))}\\
&=&\beta_a((t_1+t'_1,t_2+t'_2),(s_1+s'_1+b'_1(t_1,t'_1),s_2+s'_2+b'_2(t_2,t'_2)))\\
&=&((f^{a}(t_1+t'_1),t_2+t'_2),(s_1+s'_1+b'_1(t_1,t'_1),s_2+s'_2+b'_2(t_2,t'_2)))\\
&=&((f^{a}(t_1)+f^{a}(t'_1),t_2+t'_2),(s_1+s'_1+b'_1(f^{a}(t_1),f^{a}(t'_1)),s_2+s'_2+b'_2(t_2,t'_2)))\\
&=&((f^{a}(t_1),t_2),(s_1,s_2))+((f^{a}(t'_1),t'_2),(s'_1,s'_2))\\
&=&\beta_a((t_1,t_2),(s_1,s_2))+\beta_a((t'_1,t'_2),(s'_1,s'_2)),
\end{eqnarray*}
and
\begin{eqnarray*}
\lefteqn{\beta_a(((t_1,t_2),(s_1,s_2))\cdot ((t'_1,t'_2),(s'_1,s'_2)))}\\
&=&\beta_a((t_1\cdot f^{(2,1)}_{s_2}(t'_1),t_2\cdot f^{(1,2)}_{s_1}(t'_2)),(s_1\cdot s'_1,s_2\cdot s'_2))\\
&=&((f^{a}(t_1\cdot f^{(2,1)}_{s_2}(t'_1)),t_2\cdot f^{(1,2)}_{s_1}(t'_2)),(s_1\cdot s'_1,s_2\cdot s'_2))\\
&=&((f^{a}(t_1)\cdot f^{a}f^{(2,1)}_{s_2}(t'_1),t_2\cdot f^{(1,2)}_{s_1}(t'_2)),(s_1\cdot s'_1,s_2\cdot s'_2))\\
&=&((f^{a}(t_1)\cdot f^{(2,1)}_{s_2}f^{a}(t'_1),t_2\cdot f^{(1,2)}_{s_1}(t'_2)),(s_1\cdot s'_1,s_2\cdot s'_2))\\
&=&((f^{a}(t_1),t_2),(s_1,s_2))\cdot ((f^{a}(t'_1),t'_2),(s'_1,s'_2))\\
&=&\beta_a((t_1,t_2),(s_1,s_2))\cdot
\beta_a((t'_1,t'_2),(s'_1,s'_2)),
\end{eqnarray*}
for all $a\in \Z/(5)$ and
$((t_1,t_2),(s_1,s_2)),((t'_1,t'_2),(s'_1,s'_2))\in
T\rtimes_{\circ}S$ (recall that $T$ and $S$ are trivial left
braces). Hence $\beta$ is well-defined and it is a homomorphism of
multiplicative groups. Since $5$ is not a divisor of $|T
\rtimes_{\circ} S|$, the automorphism $\beta_{1}$ is not inner.
Hence, by Proposition~\ref{prime}, the semidirect product
$B=(T\rtimes_{\circ})\rtimes (\Z/(5))$,  via the action $a\mapsto
\beta_a$, is a prime non-simple left brace. Clearly the Sylow
subgroups of the multiplicative group of the left brace $B$ are
abelian and the result follows.
\end{proof}

\section{Yet another example}

We conclude by showing that there is more potential in the
methods introduced in Section~\ref{construction}. Namely, we explain
briefly another concrete construction of new simple left braces that
on one hand looks much more symmetric than the construction used
before, but on the other hand it can not be used to
prove Theorem~\ref{MainResult}.

Let $n>1$ be an integer. For $z\in \Z/(n)$, let $p_z$ be a prime
number, $r_{z}$ a positive integer, $V_z$ a finite dimensional
$\Z/(p_z)$-vector space and let $b_z$ be a non-singular symmetric
bilinear form over $V_z$.  Assume also that there exist elements
$f_z \in O(V_z, b_z)$ of order $p_{z-1}$. We assume that
$p_{z}\neq p_{z'}$ for $z\neq z'$. For a vector space $V$ we
denote by $M_{s,r} (V)$ the $s\times r$-matrices over $V$.
Consider the trivial left braces $T_z =M_{r_z,r_{z-1}} (V_z)$ and
$S_z= \left( \Z/(p_z) \right)^{r_z}$.  By $e^{(z)}_{1},\dots
,e^{(z)}_{r_z}$ we denote the standard basis element of $\left(
\Z/(p_z) \right)^{r_z}$.

Let $b'_z:T_z\times T_z\longrightarrow S_z$ be the symmetric
bilinear map defined by
$$b'_z((u_{i,j}),(u'_{i,j}))=\sum_{i}^{r_z}\left(\sum_{j=1}^{r_{z-1}} b_{z}(u_{i,j} , u_{i,j}')\right)e_i^{(z)},$$
for all $(u_{i,j}),(u'_{i,j})\in T_z$. Let
$f^{(z-1,z)}:S_{z-1}\longrightarrow \Aut(T_z,+)$ be the map
defined by $f^{(z-1,z)}(s)=f_s^{(z-1,z)}$ and
$$f_s^{(z-1,z)}((u_{i,j}))=\left( f_{z}^{\mu_{j}}
(u_{i,j})\right) $$ for all
$s=\sum_{j=1}^{r_{z-1}}\mu_je_{j}^{(z-1)}\in S_{z-1}$ and
$(u_{i,j})\in T_z$.

\begin{lemma}\label{orthof3}
For each $z\in \Z/(n)$, the map
$f^{(z-1,z)}:S_{z-1}\longrightarrow \Aut(T_z,+)$ is a homomorphism
of multiplicative groups and $f_s^{(z-1,z)}\in
\mathrm{O}(T_{z},b'_z)$, for all $s\in S_{z-1}$.
\end{lemma}

\begin{proof} Similar to the proof of Lemma~\ref{orthof}
\end{proof}

Let $T=T_1\times \cdots\times T_n$ and $S=S_1\times\cdots\times
S_n$ be the direct products of the left braces $T_z$'s and $S_z$'s
respectively. Let $b\colon T\times T\longrightarrow S$ be the
symmetric bilinear map defined by
$$b((t_1,\dots ,t_n),(t'_1,\dots ,t'_n))=(b'_1(t_1,t'_1),\dots ,b'_n(t_n,t'_n)),$$
for all $(t_1,\dots ,t_n),(t'_1,\dots ,t'_n)\in T$. Let
$f^{(z,z')}\colon S_z\longrightarrow \Aut(T_z,+)$ be the trivial
homomorphism for all $z'\neq z+1$. Note that the homomorphisms
$f^{(z,z')}$ satisfy (\ref{commute}). Let $\alpha\colon
(S,\cdot)\longrightarrow \Aut(T,+)$ be the map defined by
$\alpha(s_1,\dots,s_n)=\alpha_{(s_1,\dots,s_n)}$ and
\begin{eqnarray*}
\alpha_{(s_1,\dots,s_n)}(t_1,\dots ,t_n)&=&(f^{(1,1)}_{s_1}\cdots
f^{(n,1)}_{s_n}(t_1),\dots ,f^{(1,n)}_{s_1}\cdots
f^{(n,n)}_{s_n}(t_n))\\
&=&(f^{(n,1)}_{s_n}(t_1),f^{(1,2)}_{s_1}(t_2),\dots
,f^{(n-1,n)}_{s_{n-1}}(t_n)),
\end{eqnarray*}
for all $(s_1,\dots,s_n)\in S$ and $(t_1, \dots, t_n)\in T$. By
Lemma~\ref{orthof3}, and Lemma~\ref{bilinear}, we have that
$\alpha_{(s_1,\dots,s_n)}\in \mathrm{O}(T,b)$, for all
$(s_1,\dots,s_n)\in S$ and we can construct the asymmetric product
$T\rtimes_{\circ} S$ of $T$ by $S$ via $b$ and $\alpha$.

For $z\in \Z/(n)$, let $A_z=\{ ((t_1,\dots ,t_n),(s_1,\dots
,s_n))\in T\rtimes_{\circ}S\mid t_{z'}=0 \text{ and }s_{z'}=0
\text{ for all }z'\neq z \}$. By Proposition~\ref{construction1},
$A_z$ is a left ideal of $T\rtimes_{\circ}S$, in fact it is the
Sylow $p_z$-subgroup of the additive group of $T\rtimes_{\circ}S$.

\begin{theorem} \label{abundance2}
The asymmetric product $T\rtimes_{\circ} S$ is a simple left brace
if and only of $f_{z}-\id$ is bijective for all $z\in \Z/(n)$.
\end{theorem}

\begin{proof}
To prove the necessity of the stated condition it is sufficient
(and easy) to check that $J=\{ ((t_1,\dots ,t_n),(s_1,\dots
,s_n))\in T\rtimes_{\circ} S \mid t_z =(u_{i,j}^{(z)})$ such that
$u_{i,j}^{(z)} \in \text{Im}(f_{z}-\id)  \}$ is an ideal of
$T\rtimes_{\circ} S$.

The proof of the converse is similar to the proof of
Theorem~\ref{simple0} (and in some sense it is easier).
\end{proof}

\vspace{30pt}
 \noindent \begin{tabular}{llllllll}
  F. Ced\'o && E. Jespers\\
 Departament de Matem\`atiques &&  Department of Mathematics \\
 Universitat Aut\`onoma de Barcelona &&  Vrije Universiteit Brussel \\
08193 Bellaterra (Barcelona), Spain    && Pleinlaan 2, 1050 Brussel, Belgium \\
 cedo@mat.uab.cat && Eric.Jespers@vub.be \\ \\
 J. Okni\'{n}ski && \\ Institute of
Mathematics &&
\\  Warsaw University &&\\
 Banacha 2, 02-097 Warsaw, Poland &&\\
 okninski@mimuw.edu.pl &&
\end{tabular}


\begin{thebibliography}{99}
\itemsep=-2pt
\bibitem{B2} D. Bachiller, Counterexample to a conjecture about
braces, J. Algebra 453 (2016), 160--176.
\bibitem{B3} D. Bachiller, Extensions, matched products and simple braces,
J.  Pure  Appl. Algebra 222 (2018), 1670--1691.
\bibitem{BCJ} D. Bachiller, F. Ced\'o and E. Jespers, Solutions of
the Yang-Baxter equation associated with a left brace,  J. Algebra 463 (2016), 80--102.
\bibitem{BCJO} D. Bachiller, F. Ced\'o, E. Jespers and J.
Okni\'{n}ski, Iterated matched products of finite braces and
simplicity; new solutions of the Yang--Baxter equation, Trans. Amer.
Math. Soc, 370 (2018), 4881--4907.
\bibitem{BCJO2} D. Bachiller, F. Ced\'o, E. Jespers and J.
Okni\'{n}ski, Asymmetric product of left braces and simplicity; new
solutions of the Yang--Baxter equation, Communications in
Contemporary Mathematics, to appear, arXiv:1705.08493v1[math.QA].
\bibitem{BDG2} N. Ben David and Y. Ginosar, On groups of $I$-type and
involutive Yang-Baxter groups, J. Algebra 458 (2016), 197--206.
\bibitem{CCS} F. Catino, I. Colazzo and P. Stefanelli, Regular
subgroups of the afine group and asymmetric product of braces,  J.
Algebra 455 (2016), 164--182.
\bibitem{C18} F. Ced\'o, Left braces: solutions of the
Yang--Baxter equation, Advances in Group Theory and Applications 5
(2018), 33--90.
\bibitem{CJO2} F. Ced\'o, E. Jespers
and J. Okni\'{n}ski, Braces and the
Yang-Baxter equation, Commun. Math. Phys. 327 (2014), 101--116.
\bibitem{CJR} F. Ced\'o, E. Jespers and \'{A}. del R\'{\i}o, Involutive
Yang-Baxter groups, Trans. Amer. Math. Soc. 362 (2010), 2541--2558.
\bibitem{drinfeld} V. G. Drinfeld, On some unsolved problems in quantum group theory.
Quantum Groups, Lecture Notes Math. 1510, Springer-Verlag, Berlin,
1992, 1--8.
\bibitem{ESS} P. Etingof, T. Schedler and A.  Soloviev,  Set-theoretical solutions
to the quantum Yang-Baxter equation, Duke Math. J. 100 (1999), 169--209.
\bibitem{GI} T. Gateva-Ivanova, A combinatorial approach to the
set-theoretic solutions of the Yang-Baxter equation, J. Math. Phys.
45 (2004), 3828--3858.
\bibitem{GIC} T. Gateva-Ivanova and P. Cameron, Multipermutation
solutions of the Yang-Baxter equation, Comm. Math. Phys. 309 (2012),
583--621.
\bibitem{GIVdB} T. Gateva-Ivanova and M. Van den Bergh,
  Semigroups of $I$-type, J. Algebra 206 (1998), 97--112.
\bibitem{Ven1} L. Guarnieri  and L. Vendramin, Skew braces and the Yang-Baxter equation,
Math. Comp. 86 (2017), no. 307, 2519--2534.
arXiv: 1151.03171v3 [mathQA].
\bibitem{hall} P. Hall, On the system normalizers of a soluble group,
Proc. London Math. Soc. (2) 43 (1937), 507--528.
\bibitem{JObook} E. Jespers and J. Okni\'{n}ski,  Noetherian
Semigroup Algebras,
Springer, Dordrecht 2007.
\bibitem{K} C. Kassel, Quantum Groups, Springer-Verlag, 1995.
\bibitem{KSV} A. Konovalov, A. Smoktunowicz and L. Vendramin, On
skew braces and their ideals, preprint arXiv:1804.04106v2 [math.RA].
\bibitem{KS} H. Kurzweil and B. Stellmacher, The Theory of Finite Groups: an introduction, Springer, New York, 2004.
\bibitem{MW} J. MacWilliams, Orthogonal matrices over finite
fields,  Amer. Math. Monthly 76 (1969), 152--164.
\bibitem{Rump1} W. Rump, A decomposition theorem for square-free
unitary solutions of the quantum Yang-Baxter equation, Adv. Math.
193 (2005), 40--55.
\bibitem{Rump} W. Rump, Braces, radical rings, and the quantum Yang-Baxter
equation, J. Algebra 307 (2007), 153--170.
\bibitem{Rump2} W. Rump, Classification of cyclic braces, J. Pure Appl. Algebra 209 (2007),
671--685.
\bibitem{Rump5} W. Rump, Semidirect products in algebraic logic and solutions of the
quantum Yang-Baxter equation, J. Algebra Appl. 7 (2008), 471--490.
\bibitem{Rump7} W. Rump, The brace of a classical group, Note Math.
34 (2014), 115--144.
\bibitem{Rump8} W. Rump,  Classification of cyclic braces, II, Trans. Amer. Math. Soc., in print; DOI: https://doi.org/10.1090/tran/7569
\bibitem{taunt} D. R. Taunt, On $A$-groups,  Proc. Cambridge Philos. Soc. 45 (1949), 24--42.
\bibitem{Wi} R. A. Wilson, The Finite Simple Groups, Springer,
London, 2009.
\bibitem{Yang} C. N. Yang, Some exact results for the many-body
problem in one dimension with repulsive delta-function interaction,
Phys. Rev. Lett. 19 (1967), 1312--1315.
\end{thebibliography}
 \end{document}